\documentclass[a4paper]{article}
\usepackage[utf8]{inputenc}

\usepackage{subfig}
\usepackage{caption}
\usepackage{amsmath}
\usepackage{graphicx}
\usepackage[a4paper,top=2cm,bottom=2cm,left=3cm,right=3cm,marginparwidth=1.75cm]{geometry}
\usepackage{amssymb}
\usepackage{enumerate}
\usepackage{authblk}
\usepackage{amsthm}
\usepackage{subfig}
\usepackage{tikz}

\newtheorem{lem}{Lemma}

\newtheorem{theo}{Theorem}
\newcommand{\Z}{\mathbb{Z}}

\title{Monodromy Groups of Dessins d'Enfant on Rational Triangular Billiards Surfaces}

\author{
Madison Mabe\\
Lee University\\
\texttt{mmabe000@leeu.edu}

\and 
Richard A. Moy\\
Lee University\\
\texttt{rmoy@leeuniversity.edu}\\

\and

Jason Schmurr\\
Lee University\\
\texttt{jschmurr@leeuniversity.edu}\\

\and

Japheth Varlack\\
Lee University\\
\texttt{jvarla01@leeu.edu}
}

\begin{document}

\maketitle

\begin{abstract}
A \emph{dessin d’enfant}, or \emph{dessin}, is a bicolored graph embedded into a Riemann surface, and the monodromy group is an algebraic invariant of the dessin generated by rotations of edges about black and white vertices. A \emph{rational billiards surface} is a two dimensional surface that allows one to view the path of a billiards ball as a continuous path. In this paper, we classify the monodromy groups of dessins associated to rational triangular billiards surfaces.

\end{abstract}
\section*{Introduction}

A rational billiards surface is a two dimensional surface that allows one to view the path of a billiards ball as a continuous path instead of a jagged path obtained from numerous bounces off the sides of a billiards table. As one changes the shape of the billiards table, one obtains different billiards surfaces. In \cite{SMG}, the authors studied the Cayley graph associated to billiards surfaces obtained from rational triangular billiards tables. In this project, we propose modifying their approach by classifying the monodromy groups of {\it dessins d'enfant} drawn on these billiards surfaces. In particular, we prove these groups are semi-direct products of abelian groups.

\section*{Background}

\subsection*{Billiard Surfaces}

A basic definition of a billiards surface is a Riemann surface constructed from a polygon with angles in radians that are rational multiples of $\pi$ \cite{MT}.
A billiards surface is a topological construction that allows one to view the path of a billiards ball as a continuous path on a surface instead of a chaotic path of bounces off the sides of the billiards table. This concept is best illustrated via a picture involving the billiards surface associated with the rectangle. One constructs the billiards surface by beginning with one copy of the rectangle and then reflecting it a certain number of times as pictured in Figure \ref{fig:unfolding}. This is referred to as unfolding the path.
\begin{figure}[!h]
\centering
    \includegraphics[width=0.75\linewidth]{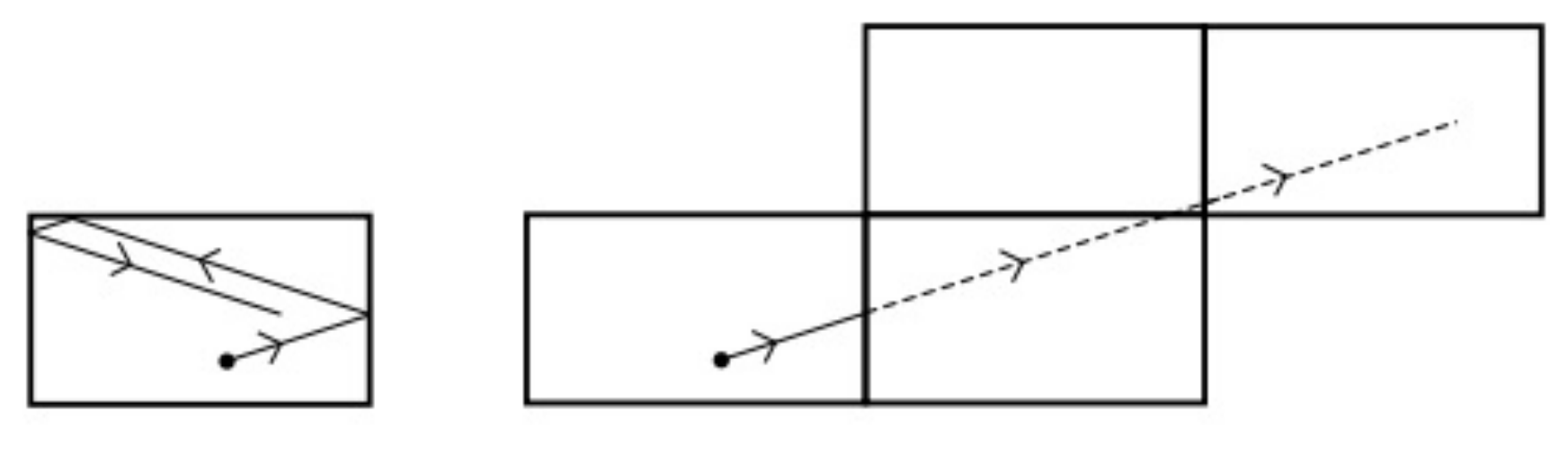}
    \caption{Unfolding a Billiard Path}
    \label{fig:unfolding}
\end{figure}
Once one has enough copies of the initial polygon, one identifies edges in the surface that are parallel and have the same orientation. In essence, this amounts to gluing together pairs of sides of our diagram. Figure \ref{fig:identify} is an example of how edges are identified in the billiards surface for the rectangle. 

\begin{figure}[!h]
\begin{minipage}{0.6\textwidth}
\centering
\includegraphics[width=\linewidth]{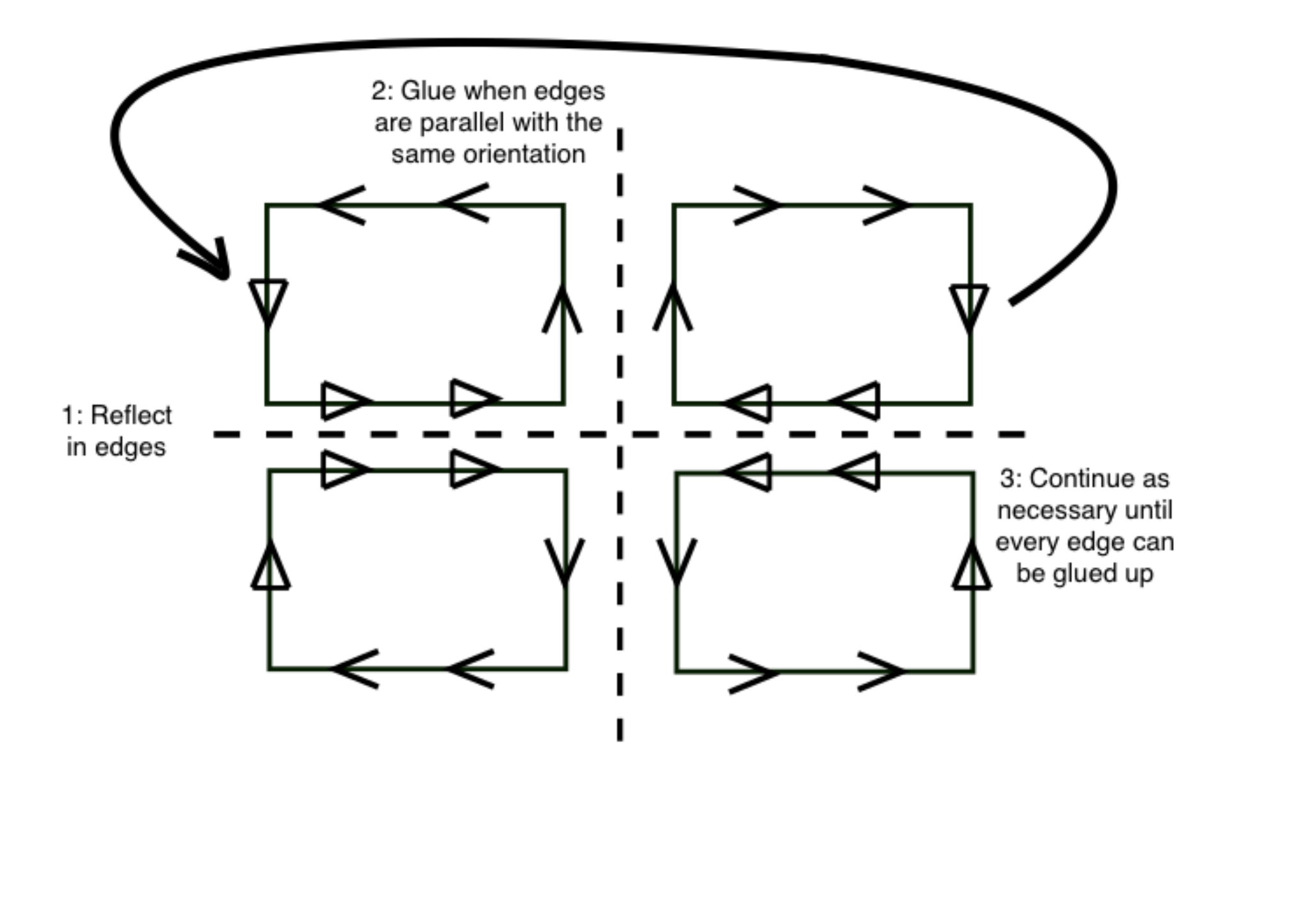}
\end{minipage}
\begin{minipage}{0.4\textwidth}
\centering
\includegraphics[width=\linewidth]{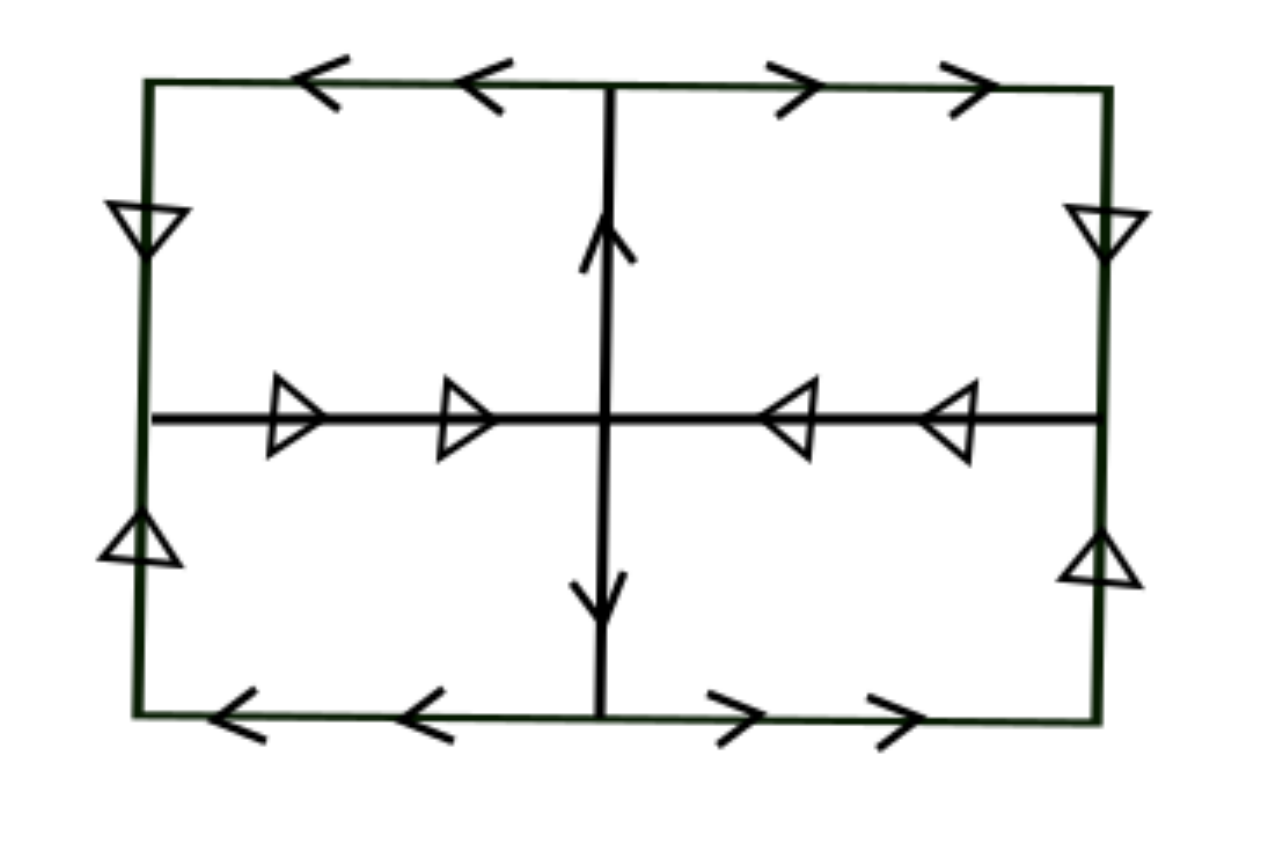}
\end{minipage}
\caption{Identifying Sides of the Billiards Diagram}
\label{fig:identify}
\end{figure}

One obtains the following 1-torus after gluing together the opposite sides of the larger rectangle.

\begin{figure}[!h]
\centering
    \includegraphics[width=0.8\linewidth]{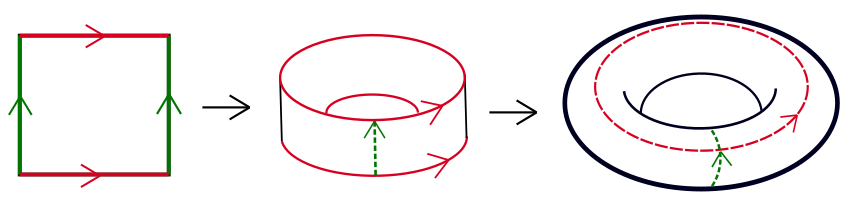}
    \caption{Gluing Sides of Billiards Diagram to Create Billiards Surface}
\end{figure}

In fact, all billiards surfaces are tori with one or more holes \cite{ZK}.

\subsection*{Graphs on Surfaces}

As studied in \cite{SMG}, one may construct a graph on the billiards surface. The Cayley graph of a billiards surface is the graph obtained by placing a vertex at the center of each polygon and edges are drawn between vertices if their corresponding polygons are adjacent. An example of the Cayley graph of the billiards surface corresponding to the $\frac{3\pi}{10},\frac{3\pi}{10},\frac{4\pi}{10}$ triangle is shown in Figure \ref{fig:cayley}.

\begin{figure}[!h]
\centering
    \includegraphics[width=0.6\linewidth]{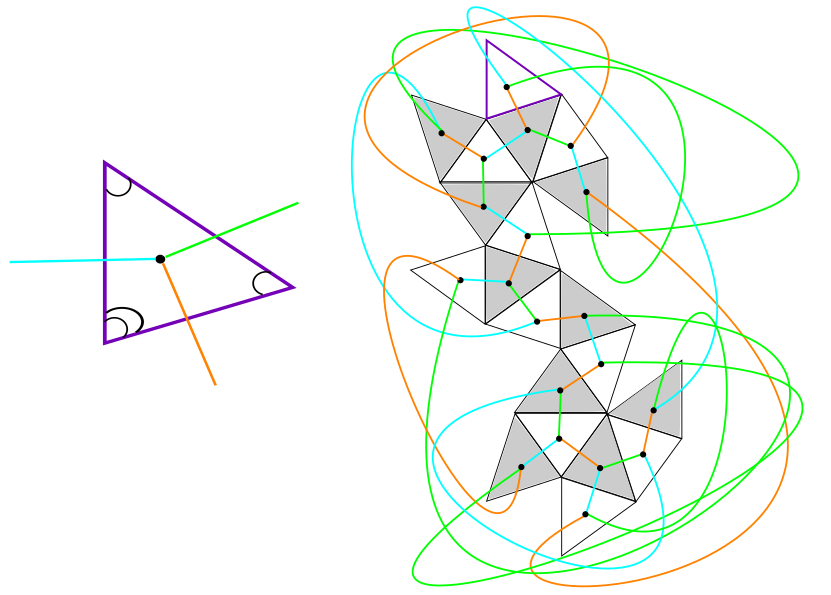}
    \caption{Cayley Graph on Billiards Surface of $\frac{3\pi}{10},\frac{3\pi}{10},\frac{4\pi}{10}$ Triangle}
    \label{fig:cayley}
\end{figure}

A billiards surface can have arbitrarily large genus. However, the genus of the Cayley graph of a triangular billiards surface is always zero or one \cite{SMG}!


Although the Cayley graphs are an interesting object to study, they lose some of the geometric information of the original billiards surface as evidenced by the fact that the genus of the Cayley graph of a triangular billiards surface is at most one. In order to create a graph that preserves more of the structure of the billiards surface, we introduce the notion of a dessin d'enfant. 


A {\it dessin d’enfant}, or simply {\it dessin}, is a connected bicolored (e.g. the vertices are one of two colors) graph equipped with a cyclic ordering of the edges (oriented counterclockwise) around each vertex \cite{LZ}. In the case of the Cayley graph, one can rearrange the order of the edges coming out of a vertex without changing the graph. However, this is not allowed in the case of dessins d'enfant. Dessins d'enfant, popularized by Grothendieck in his 1984 {\it Esquisse d'un Programme}, have been studied a great deal, because they lie at the interface of several areas of mathematics including graphs and algebraic curves. 

One important algebraic invariant of a dessin d'enfant is its monodromy group. If we have a dessin and we label the $n$ edges with the numbers $1,2,\dots, n$, we can associate the dessin with a pair of permutations $\sigma_0,\sigma_1\in S_n$, the symmetric group such that the cycles of $\sigma_0$ correspond to the cyclic ordering (read counterclockwise) of the edges around the black vertices and the cycles of $\sigma_1$  correspond to the ordering (read counterclockwise) of the edges around the white vertices. The {\it monodromy group} of a dessin with $n$ edges is $\langle \sigma_0,\sigma_1\rangle$, the group generated by $\sigma_0,\sigma_1\in S_n$.

For example, see the dessin in Figure \ref{fig:monodromy}, where we have a bicolored graph whose edges are labeled $1,2,\dots, 9$ inducing a pair of permutations $\sigma_0=(1,2,3)(4,9,8)(5,6,7),\sigma_1=(3,4,5)(1,9,6)(2,8,7)\in S_9$ associated with the black and white vertices, respectively. The monodromy group of this dessin is isomorphic to $C_3\times C_3$, where $C_3$ is the cyclic group of size $3$.

\begin{figure}[!h]
\centering
    \includegraphics[width=0.4\linewidth]{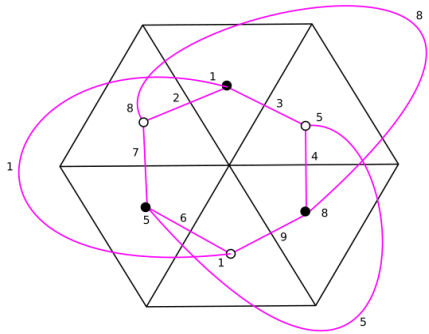}
    \caption{Determining Permutations from Dessin on Billiards Surface of Equilateral Triangle}
    \label{fig:monodromy}
\end{figure}


\subsection*{Dessin on Triangular Billiards Surfaces}

In this paper, we examine the dessin that gets drawn on a triangular rational billiard surface. Each angle of these triangles is a rational multiple of $\pi$ \cite{AI}. To denote each triangle, we use the triple notation $(p_0, p_1, p_2)$ where $p_0, p_1, p_2 \in \mathbb{N}$. To retrieve the angles of a triangle, use $(\theta_0,\theta_1,\theta_2)=(\frac{p_0\pi}{n}, \frac{p_1\pi}{n}, \frac{p_2\pi}{n})$ where $n = p_0 + p_1 + p_2$ (e.g. the triple $(1, 1, 1)$ represents the triangle with angles $(\frac{\pi}{3}, \frac{\pi}{3}, \frac{\pi}{3})$, or the equilateral triangle). Note that if $\gcd(p_0, p_1, p_2) = k \not= 1$, the triple can be reduced to $(p_0, p_1, p_2) = (\frac{p_0}{k}, \frac{p_1}{k}, \frac{p_2}{k})$. Therefore, we will only consider triples in their most reduced form, writing $T(p_0, p_1, p_2)$ for the rational triangle and $X(p_0, p_1, p_2)$ for its corresponding billiards surface.

The number of copies of $T(p_0, p_1, p_2)$ needed to form $X(p_0, p_1, p_2)$ is $2n$ \cite{AI}. A natural way of drawing the dessin on this billiard surface is drawing the black vertices on the copies with the same orientation as the original triangle and the white vertices on the copies with the opposite orientation. The edges are then drawn between two vertices if their corresponding triangles share an edge. We write $D(p_0, p_1, p_2)$ for the dessin drawn on the billiards surface $X(p_0,p_1,p_2)$. From these conditions for drawing the dessin on a triangular rational billiard surface, some basic properties can be stated:
\begin{center}
\begin{enumerate}
    \item $D(p_0, p_1, p_2)$ will have $n$ white vertices and $n$ black vertices.
    \item Each vertex in $D(p_0, p_1, p_2)$ has degree $3$.
    \item $D(p_0, p_1, p_2)$ will have $3n$ edges.
\end{enumerate}
\end{center}
Write each of the permutations $\sigma_0$ and $\sigma_1$ as a product of disjoint cycles. The cycles of $\sigma_0$ correspond to the cyclic ordering of the black vertices while the cycles of $\sigma_1$ correspond to the cyclic ordering of the white vertices. One can easily derive the following properties of $\sigma_0$ and $\sigma_1$:
\begin{center}
\begin{enumerate}
    \item $|\langle \sigma_0 \rangle| = | \langle\sigma_1 \rangle | = 3$
    \item The permutations $\sigma_0$ and $\sigma_1$ can each be written as a product of $n$ disjoint $3$-cycles.
\end{enumerate}
\end{center}

\section*{Results}
\vspace{.2 in}
The main goal of this section is the classification of all monodromy groups corresponding to triangular billiards surfaces, as stated in the following theorem.
\begin{theo}
Fix $p_0,p_1,p_2\in\mathbb{N}$ with $\gcd(p_0,p_1,p_2)=1$. Let $G= \langle \sigma_0, \sigma_1 \rangle$ be the monodromy group of the dessin $D(p_0,p_1,p_2)$ drawn on the triangular billiards surface $X(p_0,p_1,p_2)$. Setting $N=\langle \sigma_0\sigma_1, \sigma_1\sigma_0 \rangle$ and $H=\langle \sigma_0 \rangle$, we have $G = N \rtimes H$. Furthermore, if $n = p_0 + p_1 + p_2$ and $\alpha = \gcd(n, p_0p_1 - p_2^2)$, then
$$
G \cong (C_n \times C_{\frac{n}{\alpha}}) \rtimes C_3.
$$
\end{theo}

\subsection*{Notation}
The billiards surface of a triangle is generated using reflections. Fix a triangle with a black vertex and label its vertex $0$. Up to translation, half of the triangles (which have black vertices) of the billiards surface will be rotations of this fixed triangle while the other half (which have white vertices) will be reflections of those rotations. A useful way of labeling an edge on the dessin is in reference to the unique black vertex it is connected to. The black vertex labeled $m$, where $m$ is an integer and $0\le m<2n$, will be associated with the triangle rotated $\frac{2m\pi}{n}$ radians counterclockwise from the starting triangle. An example of this system for labeling the vertices is provided in Figure \ref{fig:triangle235}. 


\begin{figure}[!h]
\centering
    \includegraphics[width=0.6\linewidth]{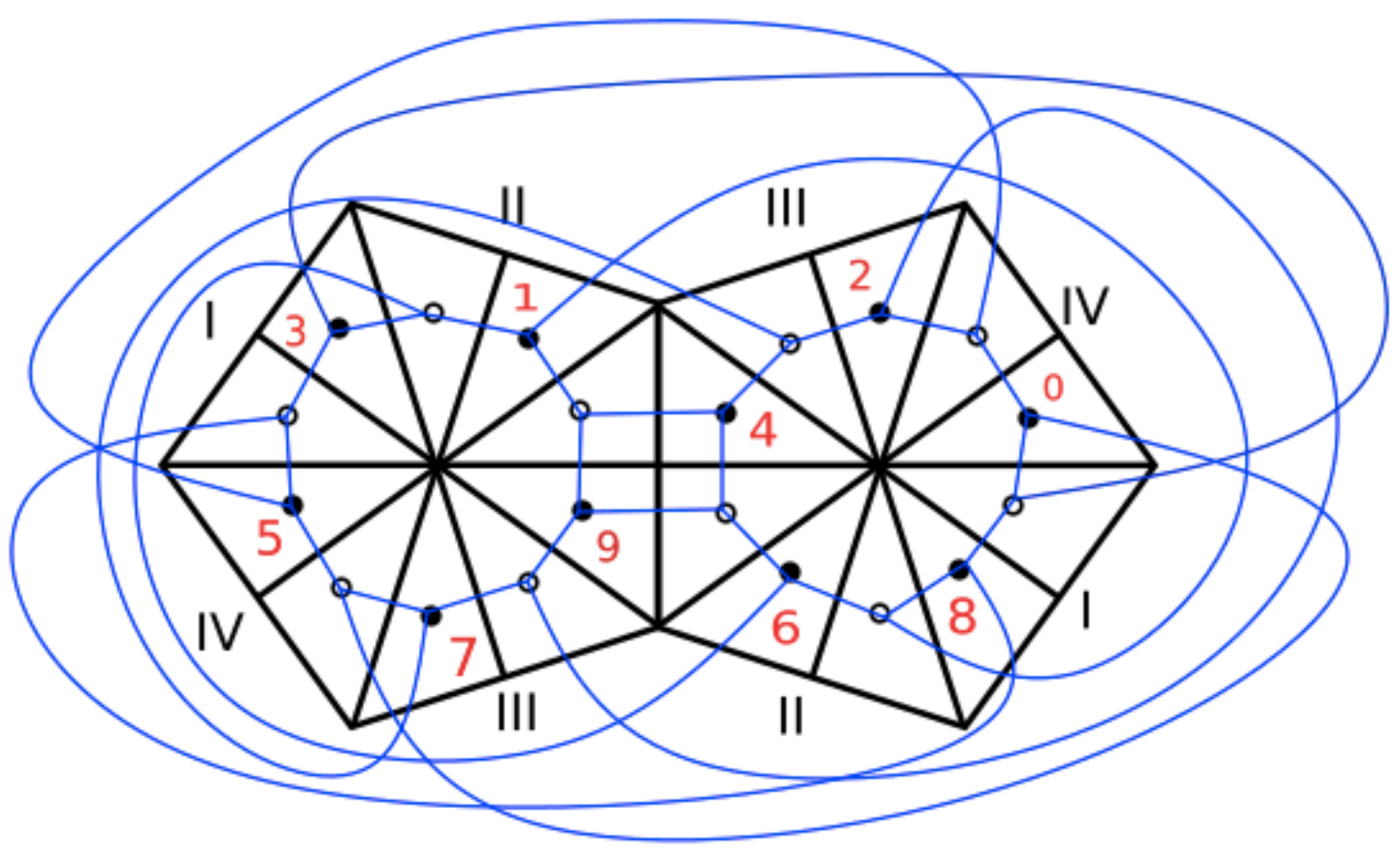}
    \caption{Labeling Black Vertices on Dessin of Billiards Surface of $\frac{2\pi}{10},\frac{3\pi}{10},\frac{5\pi}{10}$ Triangle}
    \label{fig:triangle235}
\end{figure}

\newpage
Label each side of a triangle $s_i$ where $\theta_i$ is the angle opposite of that side. Label each edge $(m,i)$ where $m$ is the black vertex incident to the edge and $s_i$ is the side of the triangle through which the edge passes. For this notation system, it is important that the angles ($\theta_0, \theta_1,\theta_2$) of a triangle  with a black vertex are ordered counterclockwise.


\begin{figure}[!h]
\centering
    \includegraphics[width=0.8\linewidth]{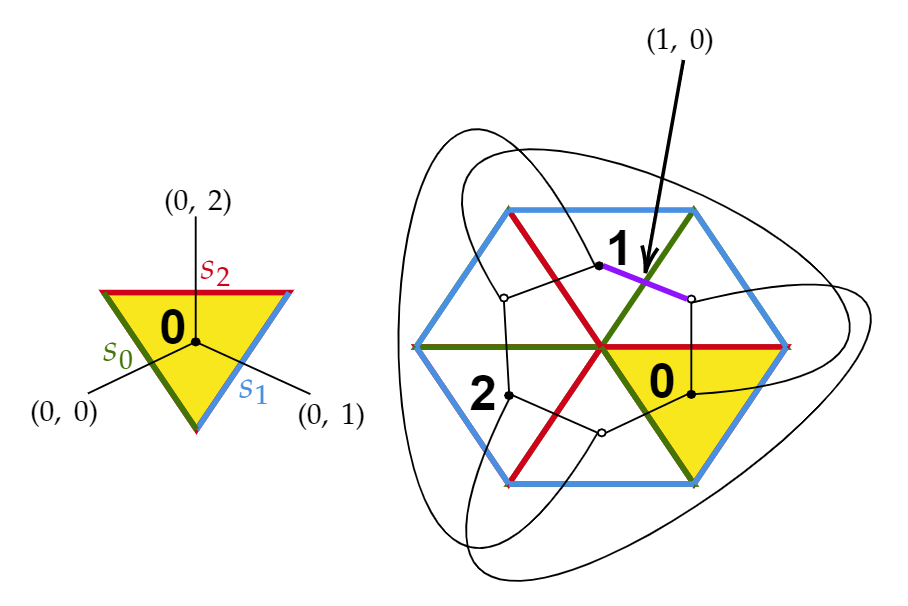}
    \caption{Labeling Edges on Dessin of Billiards Surface of Equilateral Triangle}
\end{figure}

The monodromy group acts faithfully (i.e. only the trivial permutation fixes every edge) on the set of edges according to the following formulas:

\begin{align}
    \sigma_0[(m,i)]&=
    \begin{cases}
          (m,1) &\mbox{if } i = 0\\
          (m,2) &\mbox{if } i = 1\\
          (m,0) &\mbox{if } i = 2
    \end{cases} \notag
    \\
    \sigma_1[(m,i)]&=
    \begin{cases}
          (m-p_1,2) &\mbox{if } i = 0\\
          (m-p_2,0) &\mbox{if } i = 1\\
          (m-p_0,1) &\mbox{if } i = 2
    \end{cases} \label{eqn1}
\end{align}


\begin{lem}\label{elements_commute}
The permutations $\sigma_0\sigma_1$ and $\sigma_1\sigma_0$ commute.
\end{lem}
\begin{proof}

Using \eqref{eqn1}, the permutations $\sigma_0\sigma_1$ and $\sigma_1\sigma_0$ can be computed to be the following:

\begin{align}
\sigma_0\sigma_1[(m,i)]&=
\begin{cases}
(m-p_1,0) &\mbox{if } i = 0\\
(m-p_2,1) &\mbox{if } i = 1\\
(m-p_0,2) &\mbox{if } i = 2 \end{cases} 
    \notag \\
    \sigma_1\sigma_0[(m,i)]&=
    \begin{cases}
          (m-p_2,0) &\mbox{if } i = 0\\
          (m-p_0,1) &\mbox{if } i = 1\\
          (m-p_1,2) &\mbox{if } i = 2
    \end{cases} \label{eqn2}
\end{align}
\noindent 
Furthermore, we compute

\begin{align}
    (\sigma_0\sigma_1)(\sigma_1\sigma_0)[(m,i)]&=
    \begin{cases}
          (m-p_2-p_1,0) &\mbox{if } i = 0\\
          (m-p_0-p_2,1) &\mbox{if } i = 1\\
          (m-p_1-p_0,2) &\mbox{if } i = 2
    \end{cases} =
    \begin{cases}
          (m+p_0,0) &\mbox{if } i = 0\\
          (m+p_1,1) &\mbox{if } i = 1\\
          (m+p_2,2) &\mbox{if } i = 2
    \end{cases} 
     \notag\\
    (\sigma_1\sigma_0)(\sigma_0\sigma_1)[(m,i)]&=
    \begin{cases}
          (m-p_1-p_2,0) &\mbox{if } i = 0\\
          (m-p_2-p_0,1) &\mbox{if } i = 1\\
          (m-p_0-p_1,2) &\mbox{if } i = 2
    \end{cases} 
    =
    \begin{cases}
          (m+p_0,0) &\mbox{if } i = 0\\
          (m+p_1,1) &\mbox{if } i = 1\\
          (m+p_2,2) &\mbox{if } i = 2
    \end{cases} \label{eqn3}
\end{align}
demonstrating that $\sigma_0\sigma_1$ and $\sigma_1\sigma_0$ commute.
\end{proof}

\vspace{.2 in}

\begin{lem}\label{lemma:cong}
Given a triple $(p_0, p_1, p_2)$ and $n = p_0 + p_1 + p_2$, then $p_0p_1-p_2^2 \equiv p_0p_2-p_1^2 \equiv p_1p_2-p_0^2 \pmod{n}$.
\end{lem}
\begin{proof}
Since $p_2\equiv -p_0-p_1\pmod{n}$,
$$
p_0p_1-p_2^2 \equiv p_0p_1-(-p_0-p_1)^2 = p_0p_1-2p_0p_1-p_0^2-p_1^2 = $$
$$-p_0p_1-p_0^2-p_1^2 = p_0(-p_1-p_0)-p_1^2 \equiv p_0p_2 - p_1^2 \pmod{n}.
$$
The other congruence is proven similarly.
\end{proof}
\vspace{.2 in}

\begin{lem}
Let $\alpha = \gcd(n, p_0p_1-p_2^2)$. Then $|\langle \sigma_0\sigma_1, \sigma_1\sigma_0\rangle| = \frac{n^2}{\alpha}$.
\end{lem}

\begin{proof}
We define a homomorphism $\varphi:\langle \sigma_0\sigma_1, \sigma_1\sigma_0\rangle\rightarrow (\Z\slash n\Z)^3$, where $(\Z\slash n\Z)^3$ is a $\Z\slash n\Z$ module viewed as a group, by $\varphi(\sigma_0\sigma_1)\mapsto v_1$ and $\varphi(\sigma_1\sigma_0)\mapsto v_2$ where
$$
    v_1 = 
    -\left[ 
    \begin{matrix}
     p_1 \\ p_2 \\ p_0
    \end{matrix}
    \right ],\ \ 
    v_2 = 
    -\left[ 
    \begin{matrix}
     p_2 \\ p_0 \\ p_1
    \end{matrix}
    \right ].
$$

First, we show that $\varphi$ is well-defined. Suppose that $(\sigma_0\sigma_1)^{k_1}(\sigma_1\sigma_0)^{k_2}=(\sigma_0\sigma_1)^{l_1}(\sigma_1\sigma_0)^{l_2}$ in the domain for some $k_1, k_2, l_1, l_2 \in \mathbb{Z}/n\mathbb{Z}$. Using \eqref{eqn3}, we compute
\begin{align}
    (\sigma_0\sigma_1)^{k_1}(\sigma_1\sigma_0)^{k_2}[(m,i)]&=
    \begin{cases}
          (m-k_2p_2-k_1p_1,0) &\mbox{if } i = 0\\
          (m-k_2p_0-k_1p_2,1) &\mbox{if } i = 1\\
          (m-k_2p_1-k_1p_0,2) &\mbox{if } i = 2
    \end{cases}\label{eqn4}
     \\
  \notag  (\sigma_0\sigma_1)^{l_1}(\sigma_1\sigma_0)^{l_2}[(m,i)]&=
     \begin{cases}
          (m-l_2p_2-l_1p_1,0) &\mbox{if } i = 0\\
          (m-l_2p_0-l_1p_2,1) &\mbox{if } i = 1\\
          (m-l_2p_1-l_1p_0,2) &\mbox{if } i = 2
    \end{cases} 
\end{align}
which implies that
\begin{align*}
    k_1(-p_1) + k_2(-p_2) &\equiv l_1(-p_1) + l_2(-p_2)\\
    k_1(-p_2) + k_2(-p_0) &\equiv l_1(-p_2) + l_2(-p_0)\\
    k_1(-p_0) + k_2(-p_1) &\equiv l_1(-p_0) + l_2(-p_1).
\end{align*}

Observe that $$\varphi((\sigma_0\sigma_1)^{k_1}(\sigma_1\sigma_0)^{k_2})=\left[ 
    \begin{matrix}
     k_1(-p_1) + k_2(-p_2) \\
    k_1(-p_2) + k_2(-p_0) \\
    k_1(-p_0) + k_2(-p_1)
    \end{matrix}
    \right ]=\varphi((\sigma_0\sigma_1)^{l_1}(\sigma_1\sigma_0)^{l_2})=\left[ 
    \begin{matrix}
     l_1(-p_1) + l_2(-p_2) \\
     l_1(-p_2) + l_2(-p_0)\\
     l_1(-p_0) + l_2(-p_1)
    \end{matrix}
    \right ]$$
    and thus $\varphi$ is well-defined.


We claim that $\varphi$ is injective. To show this, suppose that there are $k_1, k_2 \in \mathbb{Z}/n\mathbb{Z}$ such that $\varphi((\sigma_0\sigma_1)^{k_1}(\sigma_1\sigma_0)^{k_2})=\left[\begin{matrix} 0\\0\\0\end{matrix}  \right]\in (\mathbb{Z}/n\mathbb{Z})^3$. By examining \eqref{eqn4}, we deduce that $(\sigma_0\sigma_1)^{k_1}(\sigma_1\sigma_0)^{k_2}$ must fix every edge of the dessin. Hence, $(\sigma_0\sigma_1)^{k_1}(\sigma_1\sigma_0)^{k_2}$ is the identity element since the monodromy group acts faithfully on the edges of the dessin thus showing that $\varphi$ is injective. Therefore, $\langle \sigma_0\sigma_1, \sigma_1\sigma_0\rangle$ maps bijectively onto $\text{Span}\{v_1, v_2\}$.

Therefore, finding the order of $\text{Span}\{v_1, v_2\}$ will also give the order of $\langle \sigma_0\sigma_1, \sigma_1\sigma_0 \rangle$. We will proceed using row reduction on $[-v_1\ -v_2]$. First, recall that $\gcd(p_0, p_1, p_2) = 1$ and $p_0 + p_1 + p_2 = n$. This implies that $\gcd(p_1, p_2, n) = 1$, since a number that divides $p_1, p_2,$ and $n$ will also divide $p_0$. Thus there exist $s, t, u \in \mathbb{Z}$ such that $sp_1 + tp_2 + un = 1$ and as a consequence $sp_1 + tp_2 \equiv 1 \pmod{n}$.

Observe
$$
\left[
\begin{matrix}
 1 & 0 &0\\
 -(sp_2+tp_0)&1&0\\
 -(sp_0+tp_1)&0&1
\end{matrix}
\right]
\left[
\begin{matrix}
 p_1 & p_2\\
 p_2 & p_0 \\
 p_0 & p_1
\end{matrix}
\right]
\left [
\begin{matrix}
 s & -p_2 \\
 t & p_1
\end{matrix}
\right ]
=
\left [
\begin{matrix}
 1 & 0\\
0& p_0p_1 - p_2^2 \\
0 & p_1^2 - p_0p_2
\end{matrix}
\right ]
$$

where $$
\mathrm{det}\left(\left[
\begin{matrix}
 1 & 0 &0\\
 -(sp_2+tp_0)&1&0\\
 -(sp_0+tp_1)&0&1
\end{matrix}
\right]\right)=
\mathrm{det}\left (
\left [
\begin{matrix}
 s & -p_2 \\
 t & p_1
\end{matrix}
\right ]
\right ) = 1.
$$

Recall by Lemma \ref{lemma:cong} that $p_0p_2 - p_1^2 \equiv p_0p_1 - p_2^2 \pmod{n}$. Then this row reduced matrix is equivalent to
$$
\left [
\begin{matrix}
 1 & 0\\
 0 & p_0p_1 - p_2^2\\
 0 & -(p_0p_1 - p_2^2)
\end{matrix}
\right ].
$$
Thus, $|\text{Span}\{v_1, v_2\}| = n\cdot \frac{n}{\alpha}$ where $\alpha = \gcd(n, p_0p_1-p_2^2)$. As a result, $|\langle \sigma_0\sigma_1, \sigma_1\sigma_0\rangle| = \frac{n^2}{\alpha}$.
\end{proof}
\vspace{.2 in}

\begin{lem}\label{lemma:normal}
$N$ is a normal subgroup of $G$.
\end{lem}
\begin{proof}
Since $N = \langle \sigma_0\sigma_1, \sigma_1\sigma_0\rangle$ and $G = \langle \sigma_0, \sigma_1 \rangle$, proving $N\lhd G$ is equivalent to proving the following four statements:
\begin{itemize}
\item $\sigma_0(\sigma_1\sigma_0)\sigma_0^{-1} \in N$
\item $\sigma_1(\sigma_0\sigma_1)\sigma_1^{-1}\in N$
\item $\sigma_0(\sigma_0\sigma_1)\sigma_0^{-1}\in N$
\item $\sigma_1(\sigma_1\sigma_0)\sigma_1^{-1} \in N$
\end{itemize}
The first two statements are easy to show. To prove the third statement, observe that $\sigma_0^{-1}=\sigma_0^2$ and $\sigma_1^{-1}=\sigma_1^2$. Thus,
\begin{align*}
    \sigma_0(\sigma_0\sigma_1)\sigma_0^{-1}=\sigma_0^2\sigma_1\sigma_0^2=(\sigma_0^2\sigma_1^2)(\sigma_1^2\sigma_0^2)=(\sigma_1\sigma_0)^{-1}(\sigma_0\sigma_1)^{-1}\in N.
\end{align*}
The proof of the fourth statement is similar.
\end{proof}

\begin{lem}\label{lemma:intersection}
$N\cap H = \{id\}$
\end{lem}

\begin{proof}
Recall that $N=\langle \sigma_0\sigma_1, \sigma_1\sigma_0 \rangle$ and $H=\langle \sigma_0 \rangle$. Suppose the intersection of these groups is not trivial. Then there is an element in $N$ that is equal to $\sigma_0$ or $\sigma_0^{-1}$ which implies that $H\subset N$. Then the following should also be true:
\begin{align*}
    \sigma_1\sigma_0\sigma_0^{-1}=\sigma_1\in N
\end{align*}
Then the group $N=\langle \sigma_0,\sigma_1\rangle=G$. Since $N$ is abelian by Lemma \ref{elements_commute}, the elements $\sigma_0$ and $\sigma_1$ must commute. By examining the proof of Lemma \ref{elements_commute}, we observe that $\sigma_0$ and $\sigma_1$ commute only when $p_0\equiv p_1\equiv p_2 \pmod{n}$. This only occurs when $p_0=p_1=p_2=1$. In this case, one observes that $N=\langle \sigma_0\sigma_1, \sigma_1\sigma_0 \rangle=\langle \sigma_0\sigma_1 \rangle$ since $\sigma_0\sigma_1=\sigma_1\sigma_0$. Therefore, either $\sigma_0\sigma_1=\sigma_0$ or $(\sigma_0\sigma_1)^2=\sigma_0$ since $\sigma_0\sigma_1$ has order $3$ as an element of $G$. If $\sigma_0\sigma_1=\sigma_0$ then $\sigma_1=id$, a contradiction. If $(\sigma_0\sigma_1)^2=\sigma_0$, then $\sigma_0=\sigma_1$ which is also a contradiction. Hence, $N\cap H=\{id\}.$

\end{proof}

\begin{lem}\label{lemma:product}
$NH = G$
\end{lem}
\begin{proof}
Recall that $N = \langle \sigma_0\sigma_1 , \sigma_1\sigma_0 \rangle$, $H = \langle \sigma_0 \rangle$, and $G = \langle \sigma_0, \sigma_1 \rangle$. To show that $\sigma_0 \in NH$, choose $n = id \in N$ and $h = \sigma_0 \in H$. Then $nh = \sigma_0 \in NH$. To show that $\sigma_1 \in NH$, choose $n = \sigma_1\sigma_0 \in N$ and $h = \sigma_0^{-1} \in H$. Then $nh = \sigma_1 \in NH$. The generators of $G$ are in $NH$, so $NH = G$.
\end{proof}
\vspace{0.2in}

\begin{proof}[\textbf{Proof of Theorem 1}]\ 
\\
\\
The group $G$ is a semi direct product of subgroups $N$ and $H$ if and only if the three conditions are true:
\begin{center}
\begin{enumerate}[I.]
    \item $N \lhd G$
    \item $N \cap H = \{id\}$
    \item $NH = G$
\end{enumerate}
\end{center}
Conditions $I$, $II$, and $III$ are satisfied by Lemmas \ref{lemma:normal}, \ref{lemma:intersection}, and \ref{lemma:product} respectively. Therefore, $G$ is a semi direct product of subgroups $N$ and $H$.
\end{proof}

\section*{Future Directions}
In future research, we plan on investigating the monodromy groups of dessins d'enfant associated to rational billiards surfaces created by polygons with more than three sides.

\section*{Acknowledgements}
The first author would like to thank the Appalachian College Association for funding her research on this project through a Ledford Scholarship. The fourth author would like to thank the McNair Program at Lee University for funding his work on this project.

\bibliographystyle{plain}
\bibliography{main}

\end{document}